\documentclass[12pt,twoside,a4paper]{article}

\usepackage[noadjust]{cite}
\usepackage{amssymb,amsmath,amsthm}
\usepackage{amssymb}

\newcommand\mytitle{Couplings and primitives on topological spaces} 
\swapnumbers
\numberwithin{equation}{section}

\newtheorem{theorem}{Theorem}[section]
\newtheorem{corollary}[theorem]{Corollary}

\theoremstyle{definition}

 \mathchardef\ordinarycolon\mathcode`\:
  \mathcode`\:=\string"8000
  \begingroup \catcode`\:=\active
    \gdef:{\mathrel{\mathop\ordinarycolon}}
  \endgroup

\def\bigscpr(#1,#2){{\left(#1\nonscript \mskip2mu plus2mu \middle| \nonscript \mskip2mu
plus2mu#2\right)}}

\renewcommand\tilde{\widetilde}
\renewcommand\hat{\widehat}

\newcommand\grad{\operatorname{grad}}

\renewcommand\phi{\varphi}

\renewcommand\epsilon{\varepsilon}

\newcommand{\R}{\mathbb{R}\nonscript\hskip.03em}

\newcommand{\K}{\mathbb{K}\nonscript\hskip.03em}
\newcommand{\Z}{\mathbb{Z}\nonscript\hskip.03em}

\newcommand\cD{\mathcal D}

\newcommand\cI{\mathcal I}
\newcommand\cP{\mathcal P}

\newcommand\dupa[2]{\left\langle #1, #2 \right\rangle}

\def\formE(#1,#2){\sum_{e\in E}\int_{a_e}^{b_e} #1_e'(x)\ol{#2_e'(x)}\,dx}
\newcommand\restrict{\vphantom f\mskip1mu\vrule\mskip2mu}
\newcommand\set[2]{\bigl\{#1{;}\;#2\bigr\}}

\newcommand\ol{\overline}

\newcommand\tmo{^{-1}}

\newcommand\pd{\partial}

\newcommand\cC{\mathcal C}
\newcommand\cU{\mkern1mu\mathcal U}
\newcommand\cV{\mkern1mu\mathcal V}
\newcommand\cS{\mkern1mu\mathcal S}
\newcommand\comp{{\textnormal c}}
\newcommand\Cci{{\displaystyle 
C_{\raise0.2ex\hbox{$\scriptstyle\comp$}}^\infty}}

\newcommand{\from}{{:}\;}
\renewcommand{\from}{\colon}

\newcommand\sse{\subseteq}

\newcommand\di{\mathclose{}\,\mathrm{d}}

\newcommand\slim{\mathop{\rm s\kern.08em\mbox{\rm -}lim}} 

\newcommand\abstracttext{\noindent
For an open covering $\cU$ of a topological space and a mapping $d\from\cI\to\K$, where $\cI:=\set{(U,V)\in\cU\times\cU}{U\cap V\ne\varnothing}$, we present conditions for the
existence of a mapping $C\from\cU\to\K$ satisfying $C_V-C_U=d_{UV}$ for all $(U,V)\in\cI$. The result is applied to a Poincar\'e type theorem concerning distributional
potentials. We also put the result into the context of algebraic topology.
\vspace{8pt}

\noindent
MSC 2010: 54C20, 46F10, 55U99
\vspace{2pt}

\noindent
Keywords: topological space, simply connected, distribution, cocycle
}
\begin{document}
\title{\mytitle}

\author{J\"urgen Voigt}

\date{}

\maketitle
\begin{abstract}
\abstracttext
\end{abstract}

\section{Introduction}
\label{intro}
Let $\Omega$ be a topological space, $\cU\sse\cP(\Omega)$ an open covering of $\Omega$, and denote by 
\[
 \cI:=\set{(U,V)\in\cU\times\cU}{U\cap V\ne\varnothing}
\]
the set of intersecting pairs. A mapping $d\colon\cI\to\K$, where $\K$ denotes the field of real or complex numbers, will be called a \textbf{coupling}. A mapping $C\colon\cU\to\K$ will be called a \textbf{primitive} of $d$ if
$d_{UV}=C_V-C_U$ for all $(U,V)\in\cI$. (We write the arguments of the mappings $d$ and $C$ as subscripts.)

Clearly, if $C\colon\cU\to\K$ is any mapping, then $C$ is a primitive of the coupling $d\colon\cI\to\K$ given by
\[
 d_{UV}:=C_V-C_U\qquad((U,V)\in\cI).
\]
Note that this coupling $d$ has the properties
\begin{align}\label{eq-cc}
\begin{split}
d_{UV}= -d_{VU}\qquad&((U,V)\in\cI),\\
d_{UW}=d_{UV}+d_{VW}\qquad&(U,V,W\in\cU,\ U\cap V\cap W\ne\varnothing).
\end{split}
\end{align}
This means that the properties \eqref{eq-cc} are minimal requirements for the existence of a 
primitive for a coupling.

\begin{theorem}\label{thm-main}
Let $\Omega$ be a simply connected topological space, and let $\cU$ be a covering of\/ $\Omega$ consisting of path-connected open sets. Let $d\colon\cI\to\K$ be a coupling
satisfying the compatibility conditions \eqref{eq-cc}. Then there exists a primitive~$C$ of~$d$.
\end{theorem}

This theorem will be proved in Section~\ref{sec-pf-thm}. In order to motivate the result we indicate the following two applications.

\begin{corollary}\label{cor-1}
Let $\Omega$ and $\cU$ be as in Theorem~\ref{thm-main}. For each $U\in\cU$ let $f_U \colon U\to\K$ be a function, and suppose that for all $(U,V)\in\cI$ the function
$f_V-f_U$ is constant on $U\cap V$. 
% Then there exist a function $C\colon\cU\to\K$ and a function $f\colon\Omega\to\K$ such that $f_U=f\restrict_U+C_U$ for all $U\in \cU$.
Then there exists a function $f\colon\Omega\to\K$ such that $f_U-f$ is constant on $U$ for all $U\in \cU$.
\end{corollary}

\begin{corollary}\label{cor-2}
Let $\Omega\sse\R^n$ be a simply connected open set, and let $\cU$ be a covering of $\Omega$ consisting of path-connected 
open sets. For each $U\in\cU$ let $F_U\in\cD'(U)$ 
be a distribution. Suppose that for all $(U,V)\in\cI$ the restriction $F_V-F_U\in\cD'(U\cap V)$ to 
$\cD(U\cap V)$ is generated by a constant function. Then there exist
a mapping $C\colon\cU\to\K$ and a distribution $F\in\cD'(\Omega)$ such that $F_U=F+C_U$ on $\cD(U)$ for all $U\in\cU$.
\end{corollary}

Proofs of these corollaries will be provided in Section~\ref{sec-pf-cor}.
A proof of Corollary~\ref{cor-2}, in a special context, is given in \cite[pp.\,528--532]{Mardare2008}, and our proof of Theorem~\ref{thm-main} is inspired by this proof. 
Applying the ideas of \cite{Mardare2008} in order to prove Corollary~\ref{cor-1}, the author realised that the proofs of Corollaries~\ref{cor-1} and~\ref{cor-2} became rather similar, and this 
was the motivation for finding a context, now stated in Theorem~\ref{thm-main}, from which both corollaries can be derived.

In Section~\ref{sec-com} we sketch the application of Corollary~\ref{cor-2} in the proof of a distributional version of Poincar\'e's lemma.

In Section~\ref{sec-at} we outline how Theorem~\ref{thm-main} fits into the context of algebraic topology.

\section{Existence of primitives for couplings}
\label{sec-pf-thm}

\begin{proof}[Proof of Theorem~\ref{thm-main}]
 Let $U_0\in\cU$ and $x_0\in U_0$ be fixed throughout this proof.
Let $U\in\cU$, $x\in U$ and a path $\gamma\colon[0,1]\to\Omega$ connecting $x_0$ with
$x$, i.e.\ $\gamma$ is continuous, and $\gamma(0)=x_0$, $\gamma(1)=x$, be fixed throughout parts (i)--(iii) 
of the proof.
%  
% (i) Let $U\in\cU$ and $x\in U$. Let $\gamma\colon[0,1]\to\Omega$ be a path connecting $x_0$ with
% $x$, i.e.\ $\gamma$ is continuous, and $\gamma(0)=x_0$, $\gamma(1)=x$.

(i) As $\bigl(\gamma\tmo(V)\bigr)_{V\in\cU}$ is an open covering of $[0,1]$, there exist 
$t_0:=0<t_1<\dots<t_{n-1}<1=:t_n$ and $U_1,\dots,U_n\in\cU$ such that 
\begin{equation}\label{eq-3}
\gamma([t_{k-1},t_k])\sse U_k\qquad(k=1,\dots,n).
\end{equation}
Put
\begin{equation}\label{eq-4}
 C_{U,x,\gamma,t_1,\dots,t_{n-1}}:=\sum_{k=0}^n d_{U_kU_{k+1}},
\end{equation}
with the convention $U_{n+1}:=U$.
We are going to show that $C_{U,x,\gamma,t_1,\dots,t_{n-1}}$ does not depend on the choice of the sets $U_1,\dots,U_n$ and the intermediate points $t_1,\dots,t_{n-1}$. 
% and also, that the sum in \eqref{eq-4} is the same for any $y\in U$ and a path connecting 
% $x_0$ with $y$.

First, replacing one of the sets $U_k$ by another set $V_k\in\cU$ containing $\gamma([t_{k-1},t_k])$, we apply \eqref{eq-cc}
and obtain
\[
 d_{U_{k-1}V_k}+d_{V_kU_{k+1}} = d_{U_{k-1}U_k}+d_{U_kV_k}+d_{V_kU_k}+d_{U_kU_{k+1}}
 = d_{U_{k-1}U_k}+d_{U_kU_{k+1}}.
\]
This implies that $C_{U,x,\gamma,t_1,\dots,t_{n-1}}$ does not depend on the choice of the sets $U_k$ satisfying~\eqref{eq-3}.

Next suppose that we include an additional intermediate point $s\in(0,1)$, say $t_{k-1}<s<t_k$, for some 
$k\in\{1,\dots,n\}$. Then, observing that $\gamma([t_{k-1},s])\sse U_k$ and $\gamma([s,t_k])\sse U_k$, the 
sum in \eqref{eq-4} corresponding to the  intermediate points $t_0<\dots <t_{k-1}<s<t_k<\dots t_n$ will just contain the additional term $d_{U_kU_k}=0$, hence is unchanged.

\sloppy
Now let $s_0=0<s_1<\dots<s_{m-1}<1=s_{m}$ together with sets 
$V_1,\dots,V_m\in\cU$ satisfy the conditions
analogous to \eqref{eq-3}, and let $C_{U,x,\gamma,s_1,\dots,s_{m-1}}$ denote the corresponding sum \eqref{eq-4}. Then we can augment
one by one the original intermediate points to the set containing all the points 
$t_1,\dots,t_{n-1},s_1,\dots,s_{m-1}$, without changing the value of the sum in \eqref{eq-4}. Since we can do
the same, starting with the intermediate points $0<s_1<\dots s_{m-1}<1$, we conclude that 
\[
C_{U,x,\gamma,s_1,\dots,s_{m-1}}=C_{U,x,\gamma,t_1,\dots,t_{n-1}}=:C_{U,x,\gamma}.
\]

\fussy
(ii) Let 
$\tilde\gamma\from[0,1]\to\Omega$ 
be another path connecting $x_0$ with $x$. By hypothesis, there exists an FEP-homotopy 
$\Gamma\from[0,1]^2\to\Omega$ 
connecting $\gamma$ with $\tilde\gamma$, i.e.\ $\Gamma$ is continuous, $\gamma=\Gamma(\cdot,0)$, 
$\tilde\gamma=\Gamma(\cdot,1)$, $\Gamma(0,\cdot)=x_0$, $\Gamma(1,\cdot)=x$. We show that 
$C_{U,x,\tilde\gamma}=C_{U,x,\gamma}$. 

To show this, let $s_0\in[0,1]$. Then for the path $\Gamma(\cdot,s_0)$ there exist intermediate points 
$0<t_1<\dots <t_{n-1}<1$ and sets $U_1,\dots,U_n\in\cU$ satisfying the properties analogous to \eqref{eq-3}, and
then
\[
 C_{U,x,\Gamma(\cdot,s_0)}=C_{U,x,\Gamma(\cdot,s_0),t_1,\dots,t_{n-1}}=\sum_{k=0}^n d_{U_kU_{k+1}}.
\]
The continuity of $\Gamma$ implies that the same intermediate points, with the same sets $U_1,\dots,U_n$ can be
taken for paths $\Gamma(\cdot,s)$ with the homotopy parameter $s$ in a neighbourhood of $s_0$. This shows that the function 
$[0,1]\ni s\mapsto C_{U,x,\Gamma(\cdot,s)}$ is locally constant, hence constant, and
\[
 C_{U,x,\tilde\gamma}=C_{U,x,\Gamma(\cdot,1)}=C_{U,x,\Gamma(\cdot,0)}=C_{U,x,\gamma}=:C_{U,x}.
\]

(iii) 
Let $y\in U$. As $U$ is path-connected by hypothesis, there exists a path $\gamma_1\from[1,2]\to U$
connecting $x$ with $y$. We combine $\gamma$ with $\gamma_1$ to a path $\hat\gamma$ connecting $x_0$ with $y$. Evaluating $C_{U,x}$ and
$C_{U,y}$ by the expression \eqref{eq-3}, we see that the sums only differ by the last term in the expression for 
$C_{U,y}$, given by $d_{UU}=0$. We conclude that
\[
 C_{U,x} = C_{U,y} =:C_U.
\]

(iv) It remains to show that $C$ is a primitive of $d$, i.e.\ $d_{UV}=C_V-C_U$ for all $(U,V)\in\cI$.
Let $x\in U\cap V$, $y\in V$. Then an argument analogous to the previous step (iii) yields 
$C_V=C_{V,y,\hat\gamma}=C_{U,x,\gamma}+d_{UV}=C_U+d_{UV}$. 
\end{proof}

\section{Proofs of the corollaries}
\label{sec-pf-cor}

\begin{proof}[Proof of Corollary~\ref{cor-1}]
For all $(U,V)\in\cI$ let $d_{UV}$ be the constant satisfying $f_V=f_U+d_{UV}$ on $U\cap V$. Then clearly $d_{VU}=-d_{UV}$.
If $U,V,W\in\cU$ are such that $U\cap V\cap W\ne\varnothing$, then $f_V=f_U+d_{UV}$ and 
$f_W=f_V+d_{VW}=f_U+d_{UV}+d_{VW}$ on $U\cap V\cap W$. Since also $f_W=f_U+d_{UW}$ on 
$U\cap W\supseteq U\cap V\cap W$, one obtains $d_{UW}=d_{UV}+d_{VW}$. Hence Theorem~\ref{thm-main} implies that
there exists a primitive $C$ of $d$. Then the function $f\from\Omega\to\K$, given by
\[
 f\restrict_U:=f_U-C_U\qquad(U\in\cU)
\]
is well-defined. Indeed, if $(U,V)\in\cI$, then on $U\cap V$ one obtains
\[
 f_V-C_V=(f_U+d_{UV})-(C_U+d_{UV})=f_U-C_U. \qedhere
\]
\end{proof}

\begin{proof}[Proof of Corollary~\ref{cor-2}]
For all $(U,V)\in\cI$ let $d_{UV}$ be the constant satisfying $F_V=F_U+d_{UV}$ on $\cD(U\cap V)$. (Here the
constant $d_{UV}$ stands for the distribution generated by the constant function $d_{UV}$. More explicitly, the
equality $F_V=F_U+d_{UV}$ means that
\[
 \dupa {F_V}\phi=\dupa {F_U}\phi+\int_{U\cap V}d_{UV}\phi(x)\di x
\]
for all $\phi\in\cD(U\cap V)$.) As in the proof of Corollary~\ref{cor-1} one shows that $(d_{UV})_{(U,V)\in\cI}$ 
is a coupling satisfying \eqref{eq-cc}, and for the primitive $C$ of $d$ resulting from Theorem~\ref{thm-main} 
one concludes that $(F_U-C_U)_{U\in\cU}$ is a consistent family of distributions, i.e.
$F_U-C_U=F_V-C_V$ on $\cD(U\cap V)$ for all $(U,V)\in\cI$. It is a standard fact from the theory of 
distributions that -- using a suitable partition of unity -- one can then construct a distribution 
$F\in\cD'(\Omega)$ satisfying $F=F_U-C_U$ on $\cD(U)$ for all $U\in\cU$.
\end{proof}

\section{A distributional version of Poincar\'e's lemma}
\label{sec-com}

The present note was motivated by the following version of Poincar\'e's lemma; see \cite[Theorem~2.1]{Mardare2008}, \cite[Theorem~2.1]{Voigt2022}.

\begin{theorem}\label{thm-poinc-distr}
Let $\Omega\sse\R^n$ be a simply connected open set, and let $G=(G_1,\dots,G_n)\in\cD'(\Omega)^n$ be such that
\begin{equation}\label{eq-cc-poinc}
\pd_jG_k=\pd_kG_j\qquad (j,k\in\{1,\dots,n\}). 
\end{equation}
Then there exists $F\in\cD'(\Omega)$ such that 
$\grad F=G$.
\end{theorem}

The first step in the proof is to show that the theorem holds if $\Omega$ is an open rectangle $\prod_{j=1}^n(a_j,b_j)$; see 
\cite[pp.\,526--528]{Mardare2008} or \cite[part (i) of the proof of Theorem~2.1]{Voigt2022}. (In the latter reference it is shown for more general convex open sets.)

For each open rectangle $U\sse\Omega$ let $F_U\in\cD'(U)$ be a distribution satisfying
$\grad F_U=G$ on $\cD(U)$. If $U,V\sse\Omega$ are open rectangles satisfying $U\cap V\ne\varnothing$, then 
$\grad(F_V-F_U)=G-G=0$ on $\cD(U\cap V)$, and this implies that there exists $d_{UV}\in\K$ such that 
$F_V-F_U=d_{UV}$ on $\cD(U\cap V)$.
Choosing a covering $\cU$ of $\Omega$ by open rectangles one concludes that for each $(U,V)\in\cI$ there exists 
$d_{UV}$ as described above, and Corollary~\ref{cor-2} can be applied to provide $F$ as asserted in Theorem~\ref{thm-poinc-distr}. 
% Note that this second step of the proof of 
% Theorem~\ref{thm-poinc-distr} does not use the compatibility conditions \eqref{eq-cc-poinc}. This is in 
% contrast to the second step in \cite[part (ii) of the proof of Theorem~1.2]{Voigt2022}, where 
% \eqref{eq-cc-poinc} is used.

\section{Context algebraic topology}
\label{sec-at}

The author is indebted to Friedrich Martin Schneider (TU Bergakademie Freiberg) for pointing out how Theorem~\ref{thm-main} relates to algebraic topology and in fact
can be proved in this context. We refer to \cite[Section~3.1]{Hatcher2002} for the notions of algebraic topology used below.

\sloppy
With the covering $\cU$ of $\Omega$ one associates the \emph{nerve} (or \emph{nerve complex})
\[
 N(\cU):=\set{\cV\in\cP_{\mathrm{fin}}(\cU)}{\bigcap\cV\ne\varnothing}
\]
(where $\cP_{\mathrm{fin}}(\cU)$ stands for the collection of finite subsets of $\cU$). For this simplicial complex $\cS:=N(\cU)$ (over the set $\cU$) one considers the cochain
complex $\mathcal{C}^{\ast}(\mathcal{S},\mathbb{K}) := \bigl((C^*_n(\cS,\K))_{n\in\Z},(\pd^*_n)_{n\in\Z}\bigr)$, where $\pd^*_n\colon C^*_{n-1}(\cS,\K)\to C^*_n(\cS,\K)$ for $n\in\Z$ (and $C^*_n(\cS,\K)=\{0\}$ 
for $n\in\{-2,-3,\dots\}$). The hypothesis that $\Omega$ is simply connected and 
that (its open covering) $\mathcal{U}$ consists of path-connected sets implies (via 
standard techniques, more precisely~\cite[Theorem~3.2 (p.\,195) and 
Theorem~2A.1 (p.\,166)]{Hatcher2002} and~\cite[Theorem~8]{DeyMemoliWang2017}) that the cohomology group $H_1(\cC^*(\mathcal S,\K))$ is trivial, and this means that the set
$B_1(\cC^*(\mathcal S,\K))=\pd^*_1(C_0^*(\cS,\K))$ of $1$-coboundaries coincides with the set $Z_1(\cC^*(\mathcal S,\K))=\ker(\pd^*_2)$ of $1$-cocycles.

\fussy
Now, a coupling $d$ satisfying the compatibility conditions \eqref{eq-cc} can be interpreted as a $1$-cocycle $g$. A corresponding element $f\in C^*_0(\cS,\K)$ satisfying
$\pd^*_1f = g$ can then be interpreted as a primitive $C$ of $d$.

{
}
\bigskip

\noindent
J\"urgen Voigt\\
Technische Universit\"at Dresden\\
Fakult\"at Mathematik\\
01062 Dresden, Germany\\
{\tt juergen.voigt@tu-dresden.de
}


\frenchspacing

\begin{thebibliography}{99}

\bibitem{DeyMemoliWang2017}       
T.\,K. Dey, F. M\'emoli, Y. Wang: Topological 
analysis of nerves, Reeb spaces, mappers, and multiscale mappers. Art. no. 36, 16 pp., in: 33rd International Symposium on Computational Geometry, B.\;Aronov, M.\,J.\;Katz eds., 
% Schloss Dagstuhl - Leibniz-Zentrum for Informatik GmbH, 
Dagstuhl Publishing, Saarbr\"ucken/Wadern, Germany, 2017 (online available at http://www.dagstuhl.de/dagpub/978-3-95977-038-5).


\bibitem{Hatcher2002}
A.\;Hatcher: Algebraic topology. Cambridge University Press, Cambridge, 2002.

\bibitem{Mardare2008}
S.\;Mardare: On Poincar\'e and de Rham's theorems. Rev. Roumaine Math. Pures Appl. \textbf{53} (2008), 523--541.

\bibitem{Voigt2022}
J.\;Voigt: On the existence of distributional potentials. Math.\ Nachr.\ \textbf{296} (2023), 424--433.

\end{thebibliography}
\end{document}